\documentclass[a4paper,11pt]{amsart}
\usepackage{amssymb}
\usepackage{latexsym}
\usepackage{amsmath}
\usepackage{enumerate}

\usepackage{amsmath, hyperref}
\newtheorem{theorem}{Theorem}[section]

\newtheorem{proposition}[theorem]{Proposition}

\theoremstyle{definition}
\newtheorem{definition}[theorem]{Definition}

\newtheorem{fact}[theorem]{Fact}

\newcommand{\PA}{\mathbf{PA}}
\newcommand{\K}{\mathbf{K}}

\numberwithin{equation}{section}

\title{G\"{o}del's incompleteness theorem and the Anti-Mechanist Argument: revisited}
\author{Yong Cheng}
\address{School of Philosophy, Wuhan University, China}
\email{world-cyr@hotmail.com}
\thanks{This paper is the research result of the Humanities and Social Sciences of Ministry of Education Planning Fund project ``Research on G\"{o}del's incompleteness theorem" (project no: 17YJA72040001).
I would like to thank the fund support by the Humanities and Social Sciences of Ministry of Education Planning Fund in China. I would like to thank Peter Koellner for sending me his  drafts of \cite{Peter 18, Peter 18 second} prior to publication, and the referee for detailed helpful comments.}



\begin{document}

\begin{abstract}
This is a paper for a special issue of the journal ``Studia Semiotyczne" devoted to Stanislaw Krajewski's paper \cite{Krajewski 2019}. This paper gives some supplementary notes to Krajewski's   \cite{Krajewski 2019} on the Anti-Mechanist Arguments based on G\"{o}del's incompleteness theorem. In Section 3, we give some additional explanations to Section 4-6 in  Krajewski's \cite{Krajewski 2019} and  classify some misunderstandings of G\"{o}del's incompleteness theorem related to Anti-Mechanist Arguments. In Section 4 and 5, we give a more detailed  discussion of G\"{o}del's Disjunctive Thesis, G\"{o}del's Undemonstrability of  Consistency Thesis and the definability of natural numbers as in Section 7-8 in  Krajewski's \cite{Krajewski 2019}, describing how recent advances bear on these issues.
\end{abstract}

\maketitle

\section{Introduction}

G\"{o}del's incompleteness theorem is one of
the most remarkable and profound discoveries in the 20th century, an important milestone in the history of modern logic. G\"{o}del's incompleteness theorem has wide and profound influence on the development of logic, philosophy, mathematics, computer science and other fields, substantially shaping mathematical logic as well as foundations and philosophy of
mathematics from 1931 onward. The impact of G\"{o}del's incompleteness theorem is not confined to the community of mathematicians and logicians, and it has  been very popular and widely used outside mathematics.

G\"{o}del's incompleteness theorem raises a number of
philosophical questions concerning the
nature of mind and machine, the difference between human intelligence and machine intelligence, and the limit of machine intelligence. It is well known that Turing proposed a convincing analysis of the
vague and informal notion of ``computable" in terms of the precise mathematical notion of ``computable by a Turing machine". So we can replace the vague notion of computation with the mathematically precise notion of a Turing machine. In this paper, following Koellner in \cite{Peter 18}, we stipulate that the notion  ``the mind cannot be mechanized" means that  the mathematical outputs of the idealized human mind outstrip the mathematical outputs of any Turing  machine.\footnote{In this paper, we will not  consider the
performance of actual human minds, with their limitations and defects; but only consider the idealized human mind and look  at what it can do in principle (see \cite{Peter 18}, p. 338.).}
A popular interpretation of G\"{o}del's first incompleteness theorem ($\sf G1$) is that $\sf G1$ implies that the mind cannot be mechanized.
The Mechanistic Thesis claims that the mind can be mechanized.
In this paper, we will not examine the broad question of whether the mind can be mechanized, which has been extensively discussed in the literature (e.g.~ Penrose \cite{Penrose 89}, Chalmers \cite{Chalmers 95}, Lucas \cite{Lucas 96}, Lindstr\"{o}m \cite{Per 06}, Feferman \cite{Feferman 2009},  Shapiro \cite{Incompleteness, Shapiro Mechanism},  Koellner \cite{Peter 16, Peter 18, Peter 18 second}  and Krajewski \cite{Krajewski 2019}). Instead we will only examine the question of whether $\sf G1$ implies that the mind cannot be mechanized.

This is a paper for a special issue of Semiotic Studies devoted to Krajewski's paper \cite{Krajewski 2019}. We first give a summary of Krajewski's work in  \cite{Krajewski 2019}. In \cite{Krajewski 2019},  Krajewski gave a detailed analysis  of the alleged proof of the non-mechanical, or non-computational, character of the human mind based on G\"{o}del's incompleteness theorem. Following G\"{o}del himself and other leading logicians, Krajewski refuted the Anti-Mechanist Arguments (the Lucas Argument and the Penrose Argument), and  claimed  that  they are not implied by G\"{o}del's incompleteness theorem alone.
Moreover, Krajewski \cite{Krajewski 2019} demonstrated the inconsistency of Lucas's arithmetic and the semantic inadequacy of Penrose's arithmetic. Krajewski \cite{Krajewski 2019} also discussed two  consequences of G\"{o}del's incompleteness theorem directly related to Anti-Mechanist Arguments: our consistency is not provable (G\"{o}del's Undemonstrability of  Consistency Thesis), and we cannot define the natural numbers.
The discussion in Krajewski's paper is mainly from the philosophical perspective. However, the discussion in this paper is mainly from the logical perspective based on some recent advances on the study of G\"{o}del's incompleteness theorem and G\"{o}del's Disjunctive Thesis. Basically, we agree with Krajewski's analysis of the Anti-Mechanist Arguments and his conclusion that G\"{o}del's incompleteness theorem alone does not imply that the Anti-Mechanist Arguments hold. However, some discussions in  \cite{Krajewski 2019} are vague. Moreover, in the recent work on G\"{o}del's Disjunction Thesis one finds precise versions which can actually be \emph{proved}. The motivation of this paper is to give some supplementary notes to Krajewski's recent paper  \cite{Krajewski 2019} on the Anti-Mechanist Arguments  based on G\"{o}del's incompleteness theorem.

This paper is structured as follows. In Section 2, we review some notions and facts we will use in this paper. In Section 3, we give some supplementary notes to Section 5-6 in  Krajewski's \cite{Krajewski 2019} and classify some misunderstandings of G\"{o}del's incompleteness theorem related to Anti-Mechanist Arguments. In Section 4, we give a more detailed  discussion of  G\"{o}del's Disjunctive Thesis as in Section 7 in Krajewski's \cite{Krajewski 2019} based on recent advances of the study on G\"{o}del's Disjunctive Thesis in the literature. In Section 5, we give a more precise discussion of G\"{o}del's Undemonstrability of  Consistency Thesis and the definability of natural numbers as in Section 8 in  Krajewski's paper.

\section{Preliminaries}\label{pre sec}

In this section, we review some basic notions and facts used in this paper. Our notations are standard.  For textbooks on G\"{o}del's incompleteness theorem, we refer to \cite{Enderton 2001, metamathematics, Per 97, Smith 2007, Boolos 93}. There are some good survey papers on G\"{o}del's incompleteness theorem in the literature (see \cite{Smorynski 1977, Beklemishev 45, Kotlarski 2004, Visser 16, Cheng 19-2}).

In this paper, we focus on first order theory based on countable language, and  always assume the arithmetization of the base theory with a recursive set of non-logical constants. For a given theory $T$, we use \emph{$L(T)$} to denote  the language of $T$. For more details about arithmetization, we refer to \cite{metamathematics}.
Under the arithmetization, any formula or finite sequence of formulas  can be coded by a natural number (called the G\"{o}del number of the syntactic item).
In  this paper,  \emph{$\ulcorner\phi\urcorner$} denotes the numeral
representing the G\"{o}del number of $\phi$.

We say a set of sentences $\Sigma$  is
\emph{recursive} if the set of G\"{o}del numbers of sentences in $\Sigma$ is recursive.\footnote{For ease of exposition, we will pass back and forth between the two.} A theory $T$ is
\emph{decidable} if the set of sentences provable in $T$ is recursive; otherwise it is
\emph{undecidable}. A theory $T$ is
\emph{recursively axiomatizable} if it has a recursive set of axioms, i.e. the set of G\"{o}del numbers of axioms of $T$ is recursive. A theory $T$ is \emph{finitely axiomatizable} if it has a finite set of axioms.
A theory $T$ is
\emph{essentially undecidable} iff any recursively axiomatizable consistent extension of $T$ in the same language is undecidable.
We say a sentence $\phi$ is
\emph{independent} of $T$  if $T\nvdash \phi$ and $T\nvdash \neg\phi$. A theory $T$ is
\emph{incomplete} if there is a sentence $\phi$ in $L(T)$ which is independent of $T$; otherwise, $T$ is
\emph{complete} (i.e., for any sentence $\phi$ in $L(T)$, either $T\vdash\phi$ or $T\vdash \neg\phi$).
Informally, an interpretation of a theory $T$ in a theory $S$ is a mapping from formulas
of $T$ to formulas of $S$ that maps all axioms of $T$ to
sentences provable in $S$. If $T$ is interpretable in $S$, then all sentences provable (refutable) in $T$ are mapped, by the interpretation function, to sentences provable (refutable) in $S$. Interpretability can be accepted as a measure of strength of different theories. For the precise
definition of interpretation, we refer to \cite{Visser 11} for more details.

\begin{theorem}[\cite{undecidable}, Theorem 7, p. 22]\label{interpretable theorem}
Let $T_1$ and $T_2$ be two consistent theories such that $T_2$ is interpretable in $T_1$. If $T_2$ is essentially undecidable, then $T_1$ is also essentially undecidable.
\end{theorem}

Robinson Arithmetic $\mathbf{Q}$ was introduced in \cite{undecidable} by Tarski, Mostowski and
Robinson  as a base axiomatic theory for investigating incompleteness and undecidability.

\begin{definition}\label{def of Q}
Robinson Arithmetic $\mathbf{Q}$  is  defined in   the language $\{\mathbf{0}, \mathbf{S}, +, \cdot\}$ with the following axioms:
\begin{description}
  \item[$\mathbf{Q}_1$] $\forall x \forall y(\mathbf{S}x=\mathbf{S} y\rightarrow x=y)$;
  \item[$\mathbf{Q}_2$] $\forall x(\mathbf{S} x\neq \mathbf{0})$;
  \item[$\mathbf{Q}_3$] $\forall x(x\neq \mathbf{0}\rightarrow \exists y (x=\mathbf{S} y))$;
  \item[$\mathbf{Q}_4$]  $\forall x\forall y(x+ \mathbf{0}=x)$;
  \item[$\mathbf{Q}_5$] $\forall x\forall y(x+ \mathbf{S} y=\mathbf{S} (x+y))$;
  \item[$\mathbf{Q}_6$] $\forall x(x\cdot \mathbf{0}=\mathbf{0})$;
  \item[$\mathbf{Q}_7$] $\forall x\forall y(x\cdot \mathbf{S} y=x\cdot y +x)$.
\end{description}
\end{definition}

The theory $\mathbf{PA}$ consists of axioms $\mathbf{Q}_1$-$\mathbf{Q}_2$, $\mathbf{Q}_4$-$\mathbf{Q}_7$ in Definition \ref{def of Q} and the following axiom scheme of induction:
\[(\phi(\mathbf{0})\wedge \forall x(\phi(x)\rightarrow \phi(\mathbf{S} x)))\rightarrow \forall x  \phi(x),\]
where $\phi$ is a formula with at least one free variable $x$.

Let $\mathfrak{N}=\langle\mathbb{N}, +, \times\rangle$ denote the standard model of $\mathbf{PA}$. We say $\phi\in L(\mathbf{PA})$ is a true sentence of arithmetic if $\mathfrak{N}\models\phi$.
We define that $\mathit{Th(\mathbb{N}, +, \cdot)}$ is the set of sentence $\phi$ in $L(\mathbf{PA})$ such that $\mathfrak{N}\models\phi$. Similarly, we have the definition of $\mathit{Th(\mathbb{Z}, +, \cdot)}$, $\mathit{Th(\mathbb{Q}, +, \cdot)}$ and $\mathit{Th(\mathbb{R}, +, \cdot)}$.

We introduce a hierarchy of $L(\mathbf{PA})$-formulas called the \emph{arithmetical hierarchy} (see \cite{metamathematics, Metamathematics of First-Order
Arithmetic}). \emph{Bounded formulas}   ($\Sigma^0_0$, or $\Pi^0_0$, or $\Delta^0_0$ formula) are built from atomic
 formulas using only propositional connectives and bounded quantifiers (in the form $\forall x\leq y$ or $\exists x\leq y$). A formula is  $\Sigma^0_{n+1}$ if it has the
form $\exists x\phi$ where $\phi$ is $\Pi^0_{n}$. A formula is $\Pi^0_{n+1}$ if it has the form $\forall x\phi$ where $\phi$ is $\Sigma^0_{n}$. Thus, a $\Sigma^0_{n}$-formula has a block of $n$ alternating quantifiers, the first one
being existential, and this block is followed by a bounded formula. Similarly
for $\Pi^0_{n}$-formulas. A formula is $\Delta^0_n$ if it is equivalent to both a $\Sigma^0_{n}$ formula and a $\Pi^0_{n}$ formula.

A theory $T$ is said to be \emph{$\omega$-consistent} if there is no formula $\phi(x)$ such that $T \vdash\exists x \phi(x)$ and for any $n\in\omega$, $T \vdash\neg\phi(\bar{n})$. A theory $T$ is \emph{$1$-consistent} if there is no such formula $\phi(x)$ which is $\Delta^0_1$. A theory $T$ is
\emph{sound} iff for any formula $\phi$, if $T\vdash\phi$, then $\mathfrak{N}\models\phi$; a theory $T$ is
\emph{$\Sigma^0_1$-sound} iff for any $\Sigma^0_1$ formula $\phi$, if $T\vdash\phi$, then $\mathfrak{N}\models\phi$.


In the following, unless stated otherwise, let $T$ be a recursively axiomatizable consistent extension of $\mathbf{PA}$. There is a formal arithmetical formula $\mathbf{Proof}_{T}(x,y)$ (called G\"{o}del's proof predicate) which represents the recursive relation $Proof_{T}(x,y)$ saying that $y$ is the G\"{o}del number of a proof in $T$ of the formula with G\"{o}del number $x$. Since we will discuss general provability predicates based on proof predicates, now we give a general definition of proof predicate which is a generalization of properties of G\"{o}del's proof predicate $\mathbf{Proof}_{T}(x,y)$.

\begin{definition}\label{def of proof predicate}
We say a formula $\mathbf{Prf}_T(x, y)$ is a proof predicate of $T$ if it satisfies the following conditions:\footnote{We can say that each proof predicate represents the relation ``$y$ is the code of a proof in $T$ of a formula with G\"{o}del number $x$".}

\begin{enumerate}[(1)]
  \item $\mathbf{Prf}_T(x, y)$ is $\Delta^0_1(\mathbf{PA})$;\footnote{We say a formula $\phi$ is $\Delta^0_1(\mathbf{PA})$ if there exists a $\Sigma^0_1$ formula $\alpha$ such that $\mathbf{PA}\vdash \phi\leftrightarrow\alpha$, and there exists a $\Pi^0_1$ formula $\beta$ such that $\mathbf{PA}\vdash \phi\leftrightarrow\beta$.}
  \item $\mathbf{PA} \vdash \forall x(\mathbf{Prov}_T(x) \leftrightarrow\exists y \mathbf{Prf}_T(x, y))$;
  \item  for any $n \in\omega$ and formula $\phi, \mathbb{N}\models \mathbf{Proof}_T(\ulcorner\phi\urcorner, \overline{n}) \leftrightarrow \mathbf{Prf}_T(\ulcorner\phi\urcorner, \overline{n})$;
      \item $\mathbf{PA} \vdash \forall x\forall x^{\prime} \forall y (\mathbf{Prf}_T(x, y) \wedge \mathbf{Prf}_T(x^{\prime}, y) \rightarrow x = x^{\prime})$.
\end{enumerate}
\end{definition}

We  define the  provability predicate $\mathbf{Pr}_T(x)$ from a proof predicate $\mathbf{Prf}_T(x,y)$ by $\exists y\, \mathbf{Prf}_T(x,y)$, and   the  consistency statement $\mathbf{Con}(T)$ from a provability predicate $\mathbf{Pr}_T(x)$ by $\neg \mathbf{Pr}_T (\ulcorner \mathbf{0}\neq\mathbf{0}\urcorner)$.

\begin{description}
  \item[D1] If $T \vdash\phi$, then $T \vdash \mathbf{Pr}_T(\ulcorner\phi\urcorner)$;
  \item[D2] If $T \vdash \mathbf{Pr}_T(\ulcorner\phi \rightarrow\varphi\urcorner) \rightarrow (\mathbf{Pr}_T(\ulcorner\phi\urcorner)\rightarrow \mathbf{Pr}_T(\ulcorner\varphi\urcorner))$;
  \item[D3] $T \vdash \mathbf{Pr}_T(\ulcorner\phi\urcorner)\rightarrow \mathbf{Pr}_T(\ulcorner \mathbf{Pr}_T(\ulcorner\phi\urcorner)\urcorner)$.
\end{description}
$\mathbf{D1}$-$\mathbf{D3}$ is called the Hilbert-Bernays-L\"{o}b derivability condition.
Note that $\mathbf{D1}$ holds for any provability predicate $\mathbf{Pr}_T(x)$. We say that provability predicate $\mathbf{Pr}_T(x)$ is  \emph{standard} if it satisfies  $\mathbf{D2}$ and $\mathbf{D3}$.
In this paper, unless stated otherwise, we assume that $\mathbf{Con}(T)$ is the canonical arithmetic sentence expressing  the consistency of $T$ and $\mathbf{Con}(T)$ is formulated via a standard provability predicate.

The reflection principle for $T$, denoted by \emph{$\mathbf{Rfn}_T$}, is the schema $\mathbf{Pr}_T(\ulcorner\phi\urcorner)\rightarrow\phi$ for
every sentence $\phi$ in $L(T)$. The reflection principle for $T$ restricted to a
class of sentences $\Gamma$ will be denoted by \emph{$\Gamma$-$\mathbf{Rfn}_T$}.

Let $\alpha(x)$ be a formula in $L(T)$. We can similarly define the provability predicate and consistency statement w.r.t.~ formula $\alpha(x)$ as follows.
Define the formula $\mathbf{Prf}_{\alpha}(x,y)$ saying ``$y$ is the G\"{o}del number of a proof of the formula with G\"{o}del number $x$
from the set of all sentences satisfying $\alpha(x)$".
Define the provability predicate $\mathbf{Pr}_{\alpha}(x)$ of $\alpha(x)$ as $\exists y \mathbf{Prf}_{\alpha}(x,y)$ and the consistency statement
$\mathbf{Con}_{\alpha}(T)$ as $\neg \mathbf{Pr}_{\alpha}(\ulcorner \mathbf{0}\neq \mathbf{0}\urcorner)$.
We say that formula $\alpha(x)$ is a \emph{numeration} of $T$ if for any $n$, $T \vdash \alpha(\overline{n})$ iff $n$ is the G\"{o}del number of some sentence in $T$.

\section{Some notes on G\"{o}del-based Anti-Mechanist Arguments}

There has been a massive amount of literature on the Anti-Mechanist Arguments due primarily to
Lucas and Penrose (see Lucas \cite{Lucas 61}, Penrose \cite{Penrose 89}) which claim that $\sf G1$ shows that the human mind cannot be mechanized. The Anti-Mechanist Argument began with Nagel and Newman in \cite{Godels Proof 2}  and continued with Lucas's  publication in \cite{Lucas 61}.
Nagel and Newman's argument was criticized by Putnam in \cite{Minds and machines} and earlier by G\"{o}del (see Feferman \cite{Feferman 2009}), while Lucas's argument was much more widely criticized in the literature. See Feferman \cite{Feferman 2009} for a historical
account and  Benacerraf \cite{Benacerraf 67} for an influential criticism of Lucas.
Penrose proposed a new argument for the Anti-Mechanist Argument in \cite{Penrose 94, Penrose 11}.
Penrose's new argument is the most sophisticated
and promising Anti-Mechanist Argument  which has been extensively discussed and carefully analyzed in the literature (see Chalmers \cite{Chalmers 95}, Feferman \cite{Feferman 1995}, Lindstr\"{o}m \cite{Per 01, Per 06}, and Shapiro \cite{Incompleteness, Shapiro Mechanism},  Gaifman \cite{Gaifman 2000} and Koellner \cite{Peter 16,Peter 18,Peter 18 second}, etc).

Most philosophers and logicians believe that variants of the arguments of Lucas and Penrose are not fully convincing.
However, they do not agree so well on what is wrong with arguments of Lucas and Penrose. One strength of Krajewski's paper \cite{Krajewski 2019} is that it provides a detailed review of the history of Anti-Mechanist Arguments based on  G\"{o}del's incompleteness theorem (see Section 3 in \cite{Krajewski 2019}) and an analysis of these G\"{o}del-based Anti-Mechanist Arguments (e.g.~ Lucas's argument in Section 4 and Penrose's argument in Section 6 in \cite{Krajewski 2019}). In this section, based on Krajewski's work, we give some supplementary notes of Krajewski's Section 5-6 in \cite{Krajewski 2019}.

For us, the G\"{o}del-based Anti-Mechanist Argument comes from some misinterpretations of G\"{o}del's incompleteness theorem. To understand the source of these misinterpretations or illusions, we should first have correct interpretations of G\"{o}del's incompleteness theorem. In the following, we first review some important facts about G\"{o}del's incompleteness theorem which are helpful to clarify some misinterpretations of G\"{o}del's incompleteness theorem.

G\"{o}del proved his incompleteness theorem in \cite{Godel 1931 original proof} for a certain formal
system $\mathbf{P}$ related to Russell-Whitehead's Principia Mathematica and based on the
simple theory of types over the natural number series and the Dedekind-Peano
axioms (see \cite{Beklemishev 45}, p. 3).  G\"{o}del's original first incompleteness theorem (\cite[Theorem VI]{Godel 1931 original proof}) says that for
formal theory $T$ formulated in the language of $\mathbf{P}$ and obtained by adding a primitive recursive set of axioms to the system $\mathbf{P}$, if $T$ is $\omega$-consistent, then $T$ is incomplete.
The following theorem is a modern reformulation of G\"{o}del's first incompleteness theorem.

\begin{theorem}[G\"{o}del's first incompleteness theorem $(\sf G1)$]~
If $T$ is a recursively axiomatized extension of $\mathbf{PA}$, then there exists a G\"{o}del sentence $\mathbf{G}$ such that:
\begin{enumerate}[(1)]
  \item if $T$ is consistent, then $T\nvdash \mathbf{G}$;
  \item if $T$ is $\omega$-consistent, then $T\nvdash \neg\mathbf{G}$.
\end{enumerate}
\end{theorem}

Thus if $T$ is $\omega$-consistent, then $\mathbf{G}$ is independent of $T$ and hence
$T$ is incomplete.
If $T$ is consistent, G\"{o}del sentence $\mathbf{G}$ is a true $\Pi^0_1$ sentence of arithmetic. G\"{o}del's proof of $\sf G1$ is constructive: one can effectively find a true $\Pi^0_1$ sentence $\mathbf{G}$ of arithmetic such that $\mathbf{G}$ is independent of $T$ assuming $T$ is $\omega$-consistent. G\"{o}del\, calls this the ``incompletability or inexhaustability of mathematics".
Note that only assuming that $T$ is consistent, we can show that $\mathbf{G}$ is a true sentence of arithmetic unprovable in $T$. But it is not enough to show that $T\nvdash\neg\mathbf{G}$ only assuming that $T$ is consistent.
To show that $T\nvdash\neg\mathbf{G}$, we need a stronger condition such as ``$T$ is 1-consistent" or ``$T$ is $\Sigma^0_1$-sound".

Let $T$ be a recursively axiomatized extension of $\mathbf{PA}$. After G\"{o}del, Rosser constructed  Rosser sentence $\mathbf{R}$ (a $\Pi^0_1$ sentence) and showed that if $T$ is consistent, then  $\mathbf{R}$ is independent of $T$. Rosser improved G\"{o}del's $\sf G1$ in the sense that Rosser proved that $T$ is incomplete only assuming that ``$T$ is consistent" which is weaker than ``$T$ is 1-consistent".

In this paper, let $\langle M_n: n\in\omega\rangle$ be the list of Turing machines and $\mathit{Th(M_n)}$ be the set of sentences produced by the Turing machine $M_n$.
Let $C = \{n: \mathit{Th(M_n)}$ is a consistent theory\} and $S = \{n: \mathit{Th(M_n)}$ is a sound theory\}.
The following proposition on inconsistency and unsoundness is from \cite{Krajewski 2019}.
\begin{proposition}~\label{incon and soundness}
\begin{enumerate}[(1)]
  \item  If $F$ is a partial recursive function such that $C \subseteq dom(F)$ and $F(n)\notin \mathit{Th(M_n)}$  for any $n\in C$, then $\{F(n):n \in dom(F)\}$ is inconsistent.
  \item  If $F$ is a partial recursive function such that $S \subseteq dom(F)$ and $F(n)\notin \mathit{Th(M_n)}$ for any $n\in S$, then $\{F(n): n \in dom(F)\}$ is inconsistent.
\end{enumerate}
\end{proposition}

A natural question is: whether there exists such a function $F$ with these properties. However, the effective version of
G\"{o}del's first incompleteness theorem ($\sf EG1$) tells us that there exists a partial recursive function $F$ such that for any $n\in\omega$, if $\mathit{Th(M_n)}$ is consistent, then $F(n)$ is defined and $F(n)$ is the G\"{o}del number of a true arithmetic sentence which is not provable in $\mathit{Th(M_n)}$. Thus there exists such a function $F$ with the properties as stated in Proposition \ref{incon and soundness}.

One popular interpretation of $\sf EG1$ is: for any Turing machine $M_n, F(n)$ picks up the true sentence of arithmetic not produced by $M_n$. However, this is a misinterpretation of $\sf EG1$ which in fact says that for such a partial recursive function $F$, if $\mathit{Th(M_n)}$ is consistent, then $F(n)$ is the G\"{o}del number of a true sentence of arithmetic which is not provable in $\mathit{Th(M_n)}$. A natural question is: whether there exists an effective procedure such that we can decide whether $\mathit{Th(M_n)}$ is consistent. The answer is negative since  $C$ is a complete $\Pi^0_1$ set as Koellner points out in \cite{Peter 18}.

Krajewski \cite{Krajewski 2019} claimed that $C$ and $S$  are not recursive. However, as Krajewski \cite{Krajewski 2019} commented,  Proposition \ref{incon and soundness} on inconsistency and unsoundness  does not require that for $n\in dom(F)$, $F(n)$ is the code of a true arithmetic sentence. But
we do not see that $C$ or $S$ is not recursive from Proposition \ref{incon and soundness}. However, if we add the condition that for $n\in dom(F), F(n)\in \mathbf{Truth}\setminus \mathit{Th(M_n)}$, then we can show that $C$ and $S$ are not recursive. Let us take $C$  for example and show that $C$ is not recursive.

\begin{proposition}
$C$ is not recursive.\footnote{In fact, $C$ is a complete $\Pi^0_1$ set as Koellner points out in \cite{Peter 18}.}
\end{proposition}
\begin{proof}\label{}
Suppose $C$ is recursive. Let $A=\{F(n): n\in C\}$. Then $A$ is recursive enumerable.
Suppose $A=\mathit{Th(M_m)}$ for some $m$. Note that $A\subseteq\textbf{Truth}$, and so $A$ is consistent. By the definition of $C, m\in C$ and hence $F(m)\in A$. But, on the other hand, $F(m)\notin \mathit{Th(M_m)}=A$ which leads to a contradiction.
\end{proof}

Since $C$ is undecidable, it is impossible to effectively distinguish the case that  $\mathit{Th(M_n)}$ is consistent and the case that $\mathit{Th(M_n)}$ is not consistent.

In fact, Theorem \ref{incon and soundness} can be generalized in the following form:

\begin{theorem}\label{}
Let $P$ be any property about first order theory (i.e.~ consistency, soundness, 1-consistency, etc). Let $C = \{n: \mathit{Th(M_n)}$ has property $P$\}. Suppose $F$ is a partial recursive function satisfying the following conditions:
\begin{enumerate}[(1)]
  \item $C \subseteq dom(F)$,
  \item for each $n\in C$, $F(n)\notin \mathit{Th(M_n)}$.
\end{enumerate}
Then, $\{F(n):n \in dom(F)\}$ does not have property $P$.
\end{theorem}
\begin{proof}\label{}
Let $A=\{F(n):n \in dom(F)\}$. Suppose $A$ has property $P$. Since $F$ is partial recursive, $A$ is recursively enumerable. Suppose $A=\mathit{Th(M_k)}$ for some $k$. Since $A$ has property $P$, we have $k\in C$. Thus, $F(k)\notin \mathit{Th(M_k)}=A$ which contradicts that $F(k)\in A$.
\end{proof}

G\"{o}del announced the second incompleteness theorem $(\sf G2)$ in an abstract published in
October 1930: no consistency proof of systems such as Principia, Zermelo-Fraenkel
set theory, or the systems investigated by Ackermann and von Neumann is possible
by methods which can be formulated in these systems (see \cite{Richard Zach}, p. 431).
For a theory $T$, recall that $\mathbf{Con}(T)$ is the canonical arithmetic sentence expressing the
consistency of $T$ under G\"{o}del's recursive arithmetization of $T$. The following is a modern reformulation of $\sf G2$:

\begin{theorem}\label{G2}
Let $T$ be a recursively axiomatized extension of $\mathbf{PA}$. If $T$ is consistent, then $T\nvdash\mathbf{Con}(T)$.
\end{theorem}

From $\sf G2$, we cannot get that $\mathbf{Con}(T)$ is independent of $T$ only assuming that $T$ is consistent.
It is provable in $T$ that if $T$ is consistent, then $T\vdash \mathbf{Con}(T)\leftrightarrow \mathbf{G}$ and thus $T\nvdash\mathbf{Con}(T)$.
However, it is not provable in $T$ that if $T$ is consistent, then $T+\mathbf{Con}(T)$ is also consistent.\footnote{See \cite[Theorem 4, p. 97]{Boolos 93}  for a modal proof in $\mathbf{GL}$ of this fact using the arithmetic completeness theorem for $\mathbf{GL}$.}
So it is not enough to show that $T\nvdash \neg\mathbf{Con}(T)$  only assuming that $T$ is consistent. But we could prove that $\mathbf{Con}(T)$ is independent of $T$ by assuming that $T$ is 1-consistent which is stronger than the condition ``$T$ is consistent".\footnote{It is an easy fact that if $T$ is 1-consistent and $S$ is not a theorem of $T$, then $\mathbf{Pr}_{T}(\ulcorner S\urcorner)$ is not a theorem of $T$.}
Let 1-$\mathbf{Con}(T)$ be the sentence in $L(\mathbf{PA})$  expressing that $T$ is 1-consistent. Fact \ref{summary fact} is a summary of these results.
\begin{fact}\label{summary fact}~
Let $T$ be a recursively axiomatized consistent extension of $\mathbf{PA}$.
\begin{enumerate}[(1)]
  \item
$T\vdash\mathbf{Con}(T)\rightarrow \mathbf{Con}(T+\neg \mathbf{Con}(T))$;
\item $T\nvdash \mathbf{Con}(T)\rightarrow \mathbf{Con}(T+\mathbf{Con}(T))$;
\item  $T\vdash \mathbf{Con}(T)\rightarrow \mathbf{Con}(T+\mathbf{R})$;\footnote{Recall that $\mathbf{R}$ is the Rosser sentence.}
  \item
$T\vdash 1$-$\mathbf{Con}(T)\rightarrow \mathbf{Con}(T+ \mathbf{Con}(T))$.
\end{enumerate}
\end{fact}

An illusion of the application of G\"{o}del's incompleteness theorem is that we can add consistencies (or Out-G\"{o}deling) forever: from $\mathbf{Con}(T)$, we have $\mathbf{Con}(T+\mathbf{Con}(T))$, then $\mathbf{Con}(T+\mathbf{Con}(T+\mathbf{Con}(T)))$ and so on. However, by Fact \ref{summary fact}, this does not hold. For the iteration of adding the consistency statement (or Out-G\"{o}deling), we need a stronger condition: $T$ is 1-consistent.
The following fact shows the difference between $\mathbf{Con}(T)$ and 1-$\mathbf{Con}(T)$.
\begin{fact}[\cite{Smorynski 1977}]\label{key diff}~
Let $T$ be a recursively axiomatized consistent extension of $\mathbf{PA}$. Then $T\vdash \mathbf{Con}(T)\leftrightarrow \Pi^0_1$-$\mathbf{Rfn}_T$ and $T\vdash 1$-$\mathbf{Con}(T)\leftrightarrow \Sigma^0_1$-$\mathbf{Rfn}_T$.
\end{fact}
As a corollary of Fact \ref{key diff}, 1-$\mathbf{Con}(T)\vdash$ l-$\mathbf{Con}(T+\mathbf{Con}(T))$ (see Proposition 3 in \cite{Pudlak 99}).
Thus, if we assume 1-$\mathbf{Con}(T)$, then we can prove $\mathbf{Con}(T)$, $\mathbf{Con}(T+\mathbf{Con}(T))$,
$\mathbf{Con}(T+\mathbf{Con}(T+\mathbf{Con}(T)))$ and we can continue forever (note that the assumption 1-$\mathbf{Con}(T)$  is stronger than all these statements).


In  summary,  the differences between Rosser sentence and  G\"{o}del sentence, as well as between $\mathbf{Con}(T)$ and 1-$\mathbf{Con}(T)$ are very important. However, these differences are often overlooked in informal philosophical discussions of G\"{o}del's incompleteness theorem.

\section{G\"{o}del's disjunctive thesis}

The focus of Krajewski's  paper  \cite{Krajewski 2019} is not about G\"{o}del's Disjunctive Thesis even if he
gives a very brief discussion of G\"{o}del's Disjunctive Thesis related to the Anti-Mechanist Arguments in Section 7.
In this section, we give a more detailed discussion of G\"{o}del's Disjunctive Thesis and its relevance to the Mechanistic Thesis based on recent advances on the study of G\"{o}del's Disjunctive Thesis. This section is a summary of Koellner's papers \cite{Peter 18} and \cite{Peter 18 second}, and we follow Koellner's presentation very closely.

G\"{o}del did not argue that his incompleteness theorem implies that the
mind cannot be mechanized.  Instead, G\"{o}del argued that  his incompleteness theorem implies a weaker conclusion: G\"{o}del's Disjunctive Thesis ($\sf GD$).
\begin{description}
  \item[The first disjunct] The mind cannot be mechanized.
  \item[The second disjunct] There are absolutely undecidable statements.\footnote{In the sense that  there are
mathematical truths that cannot be proved by the idealized human mind.}
  \item[G\"{o}del's Disjunctive Thesis ($\sf GD$)] Either the first disjunct or the second disjunct holds.\footnote{The original version of $\sf GD$ was introduce by G\"{o}del in \cite{Godel 1951} (see p. 310): ``So the following disjunctive conclusion is inevitable: either mathematics is incompletable in this sense, that its evident axioms can
never be comprised in a finite rule, that is to say, the human mind
(even within the realm of pure mathematics) infinitely surpasses
the powers of any finite machine, or else there exist absolutely
unsolvable diophantine problems of the type specified (where the
case that both terms of the disjunction are true is not excluded,
so that there are, strictly speaking, three alternatives)".}
\end{description}

G\"{o}del's Disjunctive thesis ($\sf GD$) concerns the limit of mathematical knowledge and the possibility of
the existence of mathematical truths that are inaccessible to the  idealized human mind.  The first disjunct expresses an aspect of the power of the  idealized human mind, while the second disjunct expresses an aspect of its limitations.\footnote{We refer to \cite{DST}, a recent comprehensive research volume about $\sf GD$,  for more discussions of the status of $\sf GD$.}




What about G\"{o}del's view toward the first disjunct and the second disjunct?
For G\"{o}del, the first disjunct is true and the second disjunct is false; that is the mind cannot be mechanized and human mind is sufficiently powerful to capture all mathematical truths.
G\"{o}del's incompleteness theorem shows certain weaknesses and limitations of one given Turing machine. For G\"{o}del, mathematical proof is an essentially
creative activity and his incompleteness theorem indicates the creative power
of human reason.  G\"{o}del believes that   the distinctiveness of the human mind when compared to a Turing machine is evident in its ability to come up with new axioms and develop new
mathematical theories.
G\"{o}del shared Hilbert's belief expressed in 1926 in the words: ``in mathematics there is no ignoramuses, we should know and we must know" (see \cite{Reid 96}, p. 192).
Based on his rationalistic optimism, G\"{o}del believed that we are arithmetically omniscient and the second disjunct  is false.\footnote{For more discussions of the status of the second disjunct, we refer to  \cite{DST}.}
However, G\"{o}del  admits that he cannot give a convincing argument for either the first disjunct or the second disjunct. G\"{o}del thinks that the most he can claim to have established
is his  Disjunctive Thesis. For G\"{o}del, $\sf GD$  is a ``mathematically established fact" of great philosophical interest which follows from his incompleteness theorem, and it is ``entirely independent from the standpoint taken toward the foundation of mathematics" (G\"{o}del, \cite{Godel 1951}).\footnote{In the literature there is a consensus that G\"{o}del's argument for $\sf GD$ is definitive, but until now we have
no compelling evidence for or against any of the two disjuncts (Horsten and Welch, \cite{DST}).}
In the following, we give a concise overview of the current progress on G\"{o}del's disjunctive thesis based on Koellner's work in \cite{Peter 16, Peter 18, Peter 18 second}.


Let $\K$ be the set of sentences in $L(\PA)$ that the idealized human mind can know. Let $\mathbf{Truth}$ be the set of sentences in $L(\mathbf{PA})$ which are true in the standard model of arithmetic and $\textbf{Prov}$ be the set of sentences in $L(\mathbf{PA})$ which are provable in $\mathbf{PA}$. G\"{o}del refers to $\textbf{Truth}$ as objective mathematics and $\K$ as subjective mathematics. Recall that a  theory
$T$ in $L(\PA)$ is sound if $T \subseteq \mathbf{Truth}$.
In this paper, we assume  that $\K$ is sound.
However, from $\sf G1$, we have $\textbf{Prov}\subsetneq \textbf{Truth}$ since G\"{o}del's sentence  is a true sentence of arithmetic  not provable in $\PA$.\footnote{Let us take Fermat's last theorem for another example. People have shown that Fermat's last theorem is a true sentence of arithmetic but, as far as I know, it is still an open problem whether Fermat's last theorem is provable in $\PA$. So Fermat's last theorem belongs to $\K$ but it is open whether it belongs to $\textbf{Prov}$.}

Note that $\sf GD$ concerns the concepts of relative provability, absolute
provability, and truth.
Before we present the  analysis of $\sf GD$, let us first examine two key notions about provability: relative provability and absolute provability.
The notion of relative provability is well understood and we have a precise definition of relative provability in a formal system. But the notion of absolute provability is much more ambiguous and we have no unambiguous formal definition of absolute provability as far as we know.
The notion of absolute provability is intended to be intensionally different from the notion of relative provability in that absolute provability is not  conceptually connected to a formal system. In contrast to the notion of relative provability, there is little agreement  on  what principles  of the notion ``absolute provability" should be adopted.
In this paper, we identify the
notion of ``relatively provable with respect to a given formal system $F$" with the notion of ``producible by a Turing machine $M$"  (where $M$ is the Turing machine corresponding to $F$)\footnote{Note that  sentences relatively provable with respect to a given formal system $F$ can be enumerated by a Turing machine.} and we identify  the notion of ``absolute provability" with the notion of ``what the idealized human mind can know".\footnote{Williamson \cite{Williamson 16} makes the similar definition that a  mathematical hypothesis is absolutely decidable if and only if either it or its negation
can in principle be known by a normal mathematical process; otherwise it is absolutely undecidable.}
Under this assumption, $\K$ is just the set of sentences that are absolutely provable.

In this paper, we assume without loss of generality that $\mathbf{Q}\subseteq \mathit{Th(M_n)}$ such that both $\sf G1$ and $\sf G2$ apply to $\mathit{Th(M_n)}$.
For a natural number $n$, we say that a statement $\phi$ is \emph{relatively undecidable} w.r.t.~ theory $\mathit{Th(M_n)}$ for some $n$ if  $\phi\notin \mathit{Th(M_n)}$ and  $\neg\phi \notin \mathit{Th(M_n)}$. We say that a statement $\phi$ is \emph{absolutely
undecidable} if  $\phi\notin \K$ and $\neg\phi \notin \K$.
Let us first examine what the incompleteness theorem tells us about
the relationship between $\mathit{Th(M_n)}, \K$ and $\mathbf{Truth}$. 

Note that $\sf G1$ tells us that for any sufficiently strong consistent theory $F$ containing $\mathbf{Q}$, there are statements which are relatively undecidable with respect to
$F$. But as G\"{o}del argued, these statements are not absolutely undecidable; instead one can always pass to higher systems in which the sentence
in question is provable (see \cite{Godel Collected Works III}, p. 35).
For example, from $\sf G2$, $\mathbf{Con}(\PA)$ is not provable in $\PA$; but $\mathbf{Con}(\PA)$ is provable in second order arithmetic ($\mathbf{Z_2}$).
Since $\sf G2$ applies to $\mathbf{Z_2}$, the $\Pi^1_0$-truth $\mathbf{Con}(\mathbf{Z_2})$ is not provable in $\mathbf{Z_2}$. But $\mathbf{Con}(\mathbf{Z_2})$ is provable in $\mathbf{Z_3}$ (third order arithmetic) which captures the $\Pi^1_0$-truth that was missed
by $\mathbf{Z_2}$. This pattern continues up through the orders of arithmetic and
up through the hierarchy of set-theoretic systems; at each stage a missing
$\Pi^1_0$-truth is captured at the next stage (see \cite{Peter 18}, p. 347).


Now let us  examine the question of whether  the incompleteness theorem shows that $\sf GD$ holds. From the literature, we have found a natural framework $\mathbf{EA_T}$ in which we can show that  if the concepts of relative provability, absolute provability and truth satisfy some
principles, then one can give a rigorous proof of $\sf GD$, vindicating G\"{o}del's claim that $\sf GD$ is a mathematically established fact (see Koellner \cite{Peter 18}, p. 355).

Now we introduce two systems of epistemic arithmetic: $\mathbf{EA}$ and $\mathbf{EA_T}$. For the presentation of $\mathbf{EA}$ and $\mathbf{EA_T}$, we closely follow Koellner's discussion in \cite{Peter 16, Peter 18}. The first is designed to deal with $\mathit{Th(M_e)}$ and $\K$, and the second is designed to deal with $\mathit{Th(M_e)}, \K$ and $\mathbf{Truth}$.
For $\mathbf{EA_T}$, we only require a typed truth predicate.\footnote{A typed truth predicate is one that applies only to statements that do not themselves
involve the truth predicate. In contrast, a type-free truth predicate is one which also applies
to statements that themselves involve the truth predicate. The principles governing typed
truth predicates are perfectly straightforward and uncontroversial, while the principles
governing type-free truth predicates are much more delicate. See \cite{Peter 18}}
The basic system $\mathbf{EA}$ of epistemic arithmetic has axioms of arithmetic and
axioms of absolute provability, and the extended system $\mathbf{EA_T}$ has additional
axioms of typed truth.\footnote{These systems  were first introduced by Myhill \cite{Some remarks on the notion of proof}, Reinhardt \cite{Reinhardt 33,Reinhardt 32,Reinhardt 86} and Shapiro \cite{Shapiro 85}, and then investigated by many others (e.g. Horsten
\cite{defense of epistemic arithmetic}, Leitgeb \cite{On formal and informal provability}, Carlson \cite{Carlson 00}, Koellner \cite{Peter 16, Peter 18} and others).}
In $\mathbf{EA}$ and $\mathbf{EA_T}$,  $\K$ is treated as an operator rather than a predicate. From results
 in G\"{o}del \cite{interpretation}, Myhill \cite{Some remarks on the notion of proof}, Montague \cite{Richard Montague}, Thomason \cite{Thomason}, and others, if one formulates a
theory of absolute provability with $\K$ as a predicate then inconsistency may come (see \cite{Peter 18}).
The basic axioms
of absolute provability are:\footnote{The basic conditions we will impose on knowability are: (1) if the idealized human mind knows $\phi$ and  $\phi\rightarrow\psi$  then the idealized human mind knows  $\psi$; (2) if the idealized human mind knows $\phi$ then $\phi$ is true; (3) if the idealized human mind knows $\phi$ then the idealized human mind knows that the idealized human mind knows $\phi$.}
\begin{description}
  \item[K1] Universal closures of formulas of the form
$\K\phi$ where $\phi$ is a first-order validity.
  \item[K2] Universal closures of formulas of the form
$(\K(\phi\rightarrow\psi) \wedge \K\phi) \rightarrow \K\psi$.
  \item[K3] Universal closures of formulas of the form
$\K\phi \rightarrow \phi$.
  \item[K4] Universal closures of formulas of the form
$\K\phi \rightarrow \K\K\phi$.\footnote{$\mathbf{K1}$-known as logical omniscience-says that $\K$ holds of all first-order logical validities; $\mathbf{K2}$ says that $\K$ is closed
under modus ponens, and so distributes across logical derivations; $\mathbf{K3}$ says that $\K$ is correct; and $\mathbf{K4}$ says that $\K$ is absolutely self-reflective (see \cite{Peter 18}).}
\end{description}

The language $L({\mathbf{EA}})$ is $L({\mathbf{PA}})$ expanded to include an operator $\K$ that takes
formulas of $L({\mathbf{EA}})$ as arguments. The axioms of arithmetic are simply those of
$\PA$, only now the induction scheme is taken to cover all formulas in $L({\mathbf{EA}})$. For
a collection $\Gamma$ of formulas in $L({\mathbf{EA}})$, let $\K\Gamma$ denote the collection of formulas
$\K\phi$ where $\phi\in\Gamma$. The system $\mathbf{EA}$ is the theory axiomatized by $\Sigma\cup \K\Sigma$,
where $\Sigma$ consists of the axioms of $\PA$ in the language $L({\mathbf{EA}})$ and the basic
axioms of absolute provability.
The language $L({\mathbf{EA_T}})$ of $\mathbf{EA_T}$ is the language $L({\mathbf{EA}})$ augmented with a unary
predicate $T$. The system $\mathbf{EA_T}$ is the theory axiomatized by $\Sigma\cup \K\Sigma$, where
$\Sigma$ consists of the axioms of $\PA$ in the language $L({\mathbf{EA_T}})$, the basic axioms
of absolute provability (in the language $L({\mathbf{EA_T}})$), and the Tarskian axioms of
truth for the language $L({\mathbf{EA}})$.

From the incompleteness theorem, G\"{o}del made the following two claims about the relationship between  $\mathit{Th(M_e)}, \K$ and $\mathbf{Truth}$.
\begin{description}
  \item[Claim One] For any $e\in\mathbb{N}$, $\K (\mathit{Th(M_e)} \subseteq \mathbf{Truth}) \rightarrow \mathit{Th(M_e)} \subsetneqq \K$.\footnote{The informal proof of Claim One is as follows: Suppose $\K (\mathit{Th(M_e)} \subseteq \mathbf{Truth})$. Since it is knowable that $\mathit{Th(M_e)}$ is consistent, it is knowable that there is a true  sentence of arithmetic which is not provable in $\mathit{Th(M_e)}$. So $\mathit{Th(M_e)} \subsetneq \K$.}
  \item[Claim Two] Either  $\neg \exists e (\mathit{Th(M_e)} = \K)$ or  $\exists\phi (\phi\in  \mathbf{Truth} \wedge \phi \notin \K \wedge \neg\phi \notin \K)$.\footnote{The informal proof of Claim Two is as follows: Suppose $\mathit{Th(M_e)} = \K$ for some $e$. Since $\mathit{Th(M_e)}$ is R.E. but Truth is not arithmetic, $\K \subsetneq\mathbf{Truth}$. So we can find some $\phi\in  \mathbf{Truth}$ but $\phi \notin \K$ and $\neg\phi \notin \K$.}
\end{description}

G\"{o}del's Claim One is formalizable and provable in $\mathbf{EA_T}$. In fact, something stronger is provable in $\mathbf{EA}$ as the following theorem shows:

\begin{theorem}[Reinhardt, \cite{Reinhardt 32}]\label{}
Assume that $S$ includes $\mathbf{EA}$. Suppose $F(x)$ is a formula with one free variable.
\begin{enumerate}[(1)]
  \item If  for each sentence $\phi$, $S\vdash \K(F(\ulcorner\phi\urcorner) \rightarrow \phi)$.
Then there is a sentence $\varphi$ such that $S\vdash \K\varphi \wedge \K\neg F (\ulcorner\varphi\urcorner)$.
  \item If for each sentence $\phi$, $S \vdash \K(\K\phi \rightarrow F(\ulcorner\phi\urcorner))$.
Then $S \vdash \K\neg \K(\mathbf{Con}(F))$.
\end{enumerate}
\end{theorem}

From the following theorem, $\sf GD$ is also formalizable and provable in $\mathbf{EA_T}$ which confirms G\"{o}del's claim that $\sf GD$ is a mathematically established fact.\footnote{It is a little delicate to formalize $\sf GD$ in $\mathbf{EA_T}$ since $\K$ is formalized as an operator in $\mathbf{EA_T}$ and so we are prohibited from quantifying into it. For the details, we refer to Reinhardt \cite{Reinhardt 86} and Koellner \cite{Peter 16,Peter 18}.}


\begin{theorem}[Reinhardt, \cite{Reinhardt 86}]\label{}
Assume $\mathbf{EA_T}$. Then $\sf GD$ holds.
\end{theorem}

Following Reinhardt, we should distinguish three levels of the mechanistic thesis.

\begin{enumerate}[(1)]
  \item The weak mechanistic thesis ($\sf WMT$): $\exists e (\K = \mathit{Th(M_e)})$;
  \item The strong mechanistic thesis ($\sf SMT$): $\K \exists e (\K = \mathit{Th(M_e)})$;
  \item  The super strong mechanistic thesis ($\sf SSMT$): $\exists e \, \K(\K = \mathit{Th(M_e)})$.
\end{enumerate}
Note that $\sf WMT$ is just the first disjunct which says that there is
a Turing machine which coincides with the idealized human mind in the
sense that the two have the same outputs. Note that $\sf SMT$ says that the idealized human mind knows that there is a Turing machine which coincides with the idealized human mind. Note that $\sf SSMT$ says that there is a particular Turing machine
such that the idealized human mind knows that that particular machine
coincides with the idealized human mind.


Suppose $\sf WMT$ holds. Then there exists an $e^{\ast}$ such that in fact $\K = \mathit{Th(M_{e^{\ast}})}$. It might seem at first that if we know that there $\textit{is}$ such an $e^{\ast}$ then we will be able to \textit{find}, in a computable way, the indices $e$ such that $\K = \mathit{Th(M_{e})}$. But this is an illusion, as demonstrated by Rice's Theorem, which we shall now explain.

In recursion theory, the sets $\mathit{Th(M_{e})}$ are known as \textit{computably enumerable} sets. Each such set is the domain of a \textit{partial} computable function $\varphi_e$. Rice's Theorem states that for any class $C$ of partial computable functions, $\{e :\varphi_e\in C\}$ is computable iff either $C=\emptyset$ or $C$ is the class of \textit{all} partial computable functions. Now consider the set of indices that we are interested in, namely, $\{e: \K = dom(\varphi_e)\}$, that is, $\{e: \varphi_e\in C\}$ where $C=\{\varphi_e: \K=dom(\varphi_e)\}$. It follows immediately from Rice's theorem that $\{e: \K = dom(\varphi_e)\}$ is not computable.

The following theorem shows that we can prove in $\mathbf{EA_T}$ that there does not exist a particular Turing
machine such that  the idealized human mind knows that that particular Turing machine  coincides with the idealized human mind.

\begin{theorem}[Reinhardt, \cite{Reinhardt 32}]\label{SSMT inconsistent}
$\mathbf{EA_T} + \sf SSMT$ is inconsistent.
\end{theorem}

The following theorem shows that, from the viewpoint of $\mathbf{EA_T}$ it is possible that
the idealized human mind is in fact a Turing machine. From Theorem \ref{SSMT inconsistent}, it just cannot know
which one.

\begin{theorem}[Reinhardt \cite{Reinhardt 33}]\label{WMT}
$\mathbf{EA_T} + \sf WMT$ is consistent.
\end{theorem}

From Theorem \ref{WMT}, the first disjunct is not provable  in
$\mathbf{EA_T}$.
But G\"{o}del did think that one day we would be in a position to prove the first disjunct, and what was missing, as he saw it, was an adequate resolution of the paradoxes involving self-applicable concepts like the concept of truth. G\"{o}del thought
that ``[i]f one could clear up the intensional paradoxes somehow, one would
get a clear proof that mind is not machine".\footnote{This quotation is from Hao Wang's reconstruction of his conversations with G\"{o}del. See Wang \cite{Hao Wang 96}, p. 187.}


The following technical theorem from Carlson shows that, from the point of view of $\mathbf{EA_T}$, it is  possible that
the idealized human mind knows that it is a Turing machine: it just cannot know which one.

\begin{theorem}[Carlson, \cite{Carlson 00}]\label{Theorem 4.3}
$\mathbf{EA_T} + \sf SMT$ is consistent.
\end{theorem}

Now we give a summary for the question whether G\"{o}del's incompleteness theorems imply the first disjunct.
The incompleteness theorems imply that $\neg\exists e \, \K(\K =
\mathit{Th(M_e)})$. But from Theorem \ref{WMT}, it does not follow that $\neg\exists e (\K = \mathit{Th(M_e)})$; and from Theorem \ref{Theorem 4.3}, it does not even follow that $\neg K\exists e \, (\K = \mathit{Th(M_e)})$. The difference between $\exists e\, \K$ and $\K\exists e$ before $\K = \mathit{Th(M_e)}$ is essential. Assuming the principles embodied in $\mathbf{EA_T}$, it is possible to know that we are a Turing machine (i.e. $\K\exists e (\K = \mathit{Th(M_e)})$); it is
just not possible for there to be a Turing machine such that we know that we are
that Turing machine (i.e. $\exists e \,\K (\K = \mathit{Th(M_e)}))$.

Penrose proposed a new argument for the first disjunct in \cite{Penrose 94,Penrose 11}.
Penrose's new argument is the most sophisticated
and promising argument for the first disjunct. It has been extensively discussed and carefully analyzed in the literature (see Chalmers \cite{Chalmers 95}, Feferman \cite{Feferman 1995}, Lindstr\"{o}m \cite{Per 01,Per 06}, and Shapiro \cite{Incompleteness,Shapiro Mechanism},  Gaifman \cite{Gaifman 2000} and Koellner \cite{Peter 16,Peter 18 second}, etc).
The question of whether Penrose's new argument establishes the first disjunct is quite subtle.
Penrose's new argument involves treating truth as type-free, and so for the analysis and formalization of Penrose's new argument, we need to employ type-free notions of  truth. However, we now have many
type-free theories of truth and there is no consensus as to which option is best. Koellner was the first to discuss Penrose's new argument in the context of type-free truth. And he shows that when one shifts to a type-free notion of truth then one can treat $\K$ as a predicate (as a contrast, in the  context of $\mathbf{EA}$ and $\mathbf{EA_T}$, $\K$ cannot be treated as a predicate).


In the literature, Koellner  proposed the framework $\mathbf{DTK}$ which employs Feferman's type-free theory of determinate truth $\mathbf{DT}$ and some additional axioms governing $\K$ to the axioms of $\mathbf{DT}$.\footnote{For the details of the system $\mathbf{DT}$ and $\mathbf{DTK}$, see \cite{Peter 16,Peter 18 second}.}
The following results about the system $\mathbf{DTK}$ are due to Koellner.
From \cite{Peter 16, Peter 18 second}, $\mathbf{DTK}$ is consistent (see \cite[Theorem 7.14.1]{Peter 16}) and $\mathbf{DTK}$ proves $\sf GD$ (see \cite[Theorem 7.15.3]{Peter 16}).
However, the particular argument  Penrose gives for the first
disjunct fails in the context of $\mathbf{DTK}$ (see \cite[Theorem 4.1]{Peter 18 second}).
Moreover, even if we restrict the first and second disjunct to arithmetic statements, $\mathbf{DTK}$ can neither prove nor refute either the first disjunct or the second disjunct (see \cite[Theorem 7.16.1-7.16.2]{Peter 16}).  From the point of view of $\mathbf{DTK}$, it is in principle impossible
to prove or refute either disjunct. Koellner concluded  that ``Since the statements that ``the mind cannot be mechanized"
 and ``there are absolutely undecidable statements"
are independent of the natural principles governing the fundamental
concepts and, moreover, are independent of any plausible principles
in sight, it seems likely that these statements are themselves ``absolutely
undecidable"" (Koellner, \cite{Peter 18 second}, p. 469).\footnote{Koellner concluded in \cite{Peter 18 second} with a disjunctive conclusion of his
own: ``Either the statements that ``the mind cannot be mechanized"
and ``there are absolutely undecidable statements" are indefinite (as the philosophical
critique maintains) or they are definite and the above results and considerations
provide evidence that they are about as good examples of ``absolutely
undecidable" propositions as one might find" (\cite{Peter 18 second}, p. 480).}


In our previous discussion of $\sf GD$, the first disjunct and the second disjunct, we identified absolutely undecidability with knowability of the idealized human mind and define that $\phi$ is absolutely undecidable if $\phi\notin \K$ and $\neg\phi\notin \K$. Under this framework, the second disjunct  is equivalent to ``$\K$ is not complete". Under the assumption that $\K \subseteq \textbf{Truth}$, the second disjunct  is equivalent to ``$\K \subsetneq \textbf{Truth}$". However, $\sf G1$ only tells us that $\textbf{Prov}\subsetneq\textbf{Truth}$, and it does not tell us that $\K \subsetneq \textbf{Truth}$.

Another  natural informal definition of  absolutely undecidability is: $\phi$ is absolutely undecidable if there is no   consistent extension $T$ of $\mathbf{ZFC}$ with well-justified axioms such that $\phi$ is provable in $T$.  In this paper, we  focus on whether G\"{o}del's incompleteness theorem implies that the human mind cannot be mechanized. In philosophy of set theory, there are extensive discussions about wether there exists an absolutely undecidable statement in set theory. For a detailed discussion of  the question of absolutely undecidability in set theory and especially whether the Continuum Hypothesis is absolutely undecidable, we refer to Koellner \cite{Peter 06}.



\section{G\"{o}del's Undemonstrability of  Consistency Thesis and the definability of natural numbers}

In Section 8, Krajewski    \cite{Krajewski 2019} discussed  two consequences of G\"{o}del's incompleteness theorem directly related to the Anti-Mechanist Arguments: G\"{o}del's Undemonstrability of  Consistency Thesis  and the undefinability of natural numbers.
For us, Krajewski's discussion on these two consequences is mainly philosophical and not very precise. In this section, we want to give a more precise   logical analysis of  G\"{o}del's Undemonstrability of Consistency Thesis and the undefinability of natural numbers.

Let us first examine the definability of natural numbers.
As a consequence of G\"{o}del's incompleteness theorem, Krajewski    \cite{Krajewski 2019} claimed that we can not define the natural numbers in the sense that there is not a complete axiomatic system which fully characterizes all truths about natural numbers. We give some supplementary notes to make this point more precise.

Firstly,
whether a theory about natural numbers is complete depends on the language of the theory. In the languages  $\, L(\mathbf{0}, \mathbf{S})$, $L(\mathbf{0}, \mathbf{S}, <)$  and  $L(\mathbf{0}, \mathbf{S}, <, +)$, there are, respectively, recursively axiomatized complete arithmetic theories (see Section 3.1-3.2 in \cite{Enderton 2001}).
For example, Presburger arithmetic is a complete theory of the arithmetic of addition in the language  $L(\mathbf{0}, \mathbf{S}, +)$ (see Theorem 3.2.2 in \cite{metamathematics}, p. 222). However, if a recursively axiomatized theory contains enough information about addition and multiplication, then it is incomplete and hence it must miss some truths about arithmetic. For example, any recursively axiomatized consistent extension of $\mathbf{Q}$ is incomplete. Thus, in Krajewski's sense, we can not define the natural numbers in any recursively axiomatized consistent extension of $\mathbf{Q}$.


Secondly, if we discuss the definability of a set with respect to a structure, then the definability of natural numbers depends on the structure we talk about.
It is well known that $\mathbb{N}$ is definable in $(\mathbb{Z}, +, \cdot)$ and $(\mathbb{Q}, +, \cdot)$ (see Chapter XVI in \cite{Epstein 2011}), and  $\mathit{Th(\mathbb{N}, +, \cdot)}$ is interpretable in $\mathit{Th(\mathbb{Z}, +, \cdot)}$ and $\mathit{Th(\mathbb{Q}, +, \cdot)}$.
Since $\mathit{Th(\mathbb{N}, +, \cdot)}$ is undecidable,\footnote{I.e.~ there does not exist an effective algorithm such that given any sentence $\phi$ in $L(\mathbf{PA})$, we can effectively decide whether $\langle\mathbb{N}, +, \cdot\rangle\models \phi$ or not.} by Theorem \ref{interpretable theorem}, $\mathit{Th(\mathbb{Z}, +, \cdot)}$ and $\mathit{Th(\mathbb{Q}, +, \cdot)}$  are all undecidable and hence not recursive axiomatizable. But $\mathit{Th(\mathbb{R}, +, \cdot)}$ is a decidable, recursively axiomatizable  theory (even if not finitely axiomatizable) and  $\mathit{Th(\mathbb{R}, +, \cdot)}=\mathbf{RCF}$ (the theory of real closed field) (see \cite{Epstein 2011}, p. 320-321).
As a corollary, $\mathbb{N}$ is not definable in the structure $\langle\mathbb{R}, +, \cdot\rangle$ (if  $\mathbb{N}$ is  definable in  $\langle\mathbb{R}, +, \cdot\rangle$, then $\mathit{Th(\mathbb{N}, +, \cdot)}$ is interpretable in $\mathit{Th(\mathbb{R}, +, \cdot)}$ and thus, by Theorem \ref{interpretable theorem}, $\mathit{Th(\mathbb{R}, +, \cdot)}$ is undecidable which leads to a contradiction).
In  summary, if we consider matters of definability relative to the base structure, then
whether the set of natural numbers is definable depends on the base structure: $\mathbb{N}$ is definable in $\langle\mathbb{Z}, +, \cdot\rangle$ and $\langle\mathbb{Q}, +, \cdot\rangle$, but $\mathbb{N}$ is not definable in  $\langle\mathbb{R}, +, \cdot\rangle$.


Now we examine G\"{o}del's Undemonstrability of Consistency Thesis (i.e.~ ${\sf G2}$).
The intensionality of G\"{o}del sentence and the consistency sentence has been widely discussed in the literature (e.g.~ Feferman \cite{Feferman 60}, Halbach-Visser \cite{Halbach 2014a, Halbach 2014b}, Visser \cite{Visser 11}). Halbach and Visser examined the sources of intensionality in the construction of self referential sentences of arithmetic  in \cite{Halbach 2014a, Halbach 2014b} and argued that corresponding to the three stages of the construction of self referential sentences of arithmetic, there are at least three sources
of intensionality: coding, expressing a property and  self-reference.
Visser \cite{Visser 11} located three
sources of indeterminacy in the formalization of a consistency statement for a
theory $T$:
\begin{enumerate}[(I)]
  \item the choice of a proof system;
  \item  the choice of a way of numbering;
  \item the choice of a specific formula numerating the axiom set of $T$.
\end{enumerate}
In summary, the intensional nature ultimately traces back to the various parameter choices that one has to make in arithmetizing the provability predicate. That is the source of both the intensional nature of the G\"{o}del sentence and the consistency sentence.

For a consistent theory $T$, we say that $\sf G2$ holds for $T$ if the consistency statement of $T$ is not provable in $T$. However, this definition is vague, and whether $\sf G2$ holds for $T$ depends on how we formulate the consistency statement. We refer to this phenomenon as the intensionality of $\sf G2$.
Both mathematically and philosophically, $\sf G2$ is  more problematic than $\sf G1$.
The difference between $\sf G1$ and $\sf G2$ is that in the case of $\sf G1$ we are mainly interested in the fact that it shows that \emph{some} sentence is undecidable if $\mathbf{PA}$ is $\omega$-consistent. We make no claim to the effect that that sentence ``really" expresses what we would express by saying ``$\mathbf{PA}$ cannot prove this sentence".\footnote{I would also like to thank the referee for pointing out this difference between $\sf G1$ and $\sf G2$.} But in the case of $\sf G2$ we are also interested in the content of the statement.

The status of $\sf G2$ is essentially different from $\sf G1$ due to the intensionality of $\sf G2$.
We can say that ${\sf G1}$ is extensional in the sense that we can construct a concrete independent mathematical statement without referring to  arithmetization and provability predicate.
However, ${\sf G2}$ is intensional and ``whether ${\sf G2}$ holds for $T$" depends on varied factors as we will discuss.

In the following, we give a very brief discussion of the intensionality  of $\sf G2$ (we refer to \cite{Cheng 19-2} for more details).
In this section, unless otherwise stated, we make the following assumptions:
\begin{enumerate}[(1)]
  \item The theory $T$ is a recursively axiomatized consistent  extension of  $\mathbf{Q}$;
  \item The canonical arithmetic formula to express the consistency of  $T$ is $\mathbf{Con}(T)\triangleq \neg \mathbf{Pr}_T(\ulcorner\mathbf{0}\neq \mathbf{0}\urcorner)$;
  \item  The canonical numbering we use is G\"{o}del's numbering;
  \item The provability predicate we use is standard;
  \item The formula numerating the axiom set of $T$ is $\Sigma^0_1$.
\end{enumerate}

Based on works  in the literature, we argue that ``whether $\sf G2$ holds for $T$" depends on the following factors:
\begin{enumerate}[(1)]
\item the choice of the base theory $T$;
\item the choice of a provability predicate;
    \item the choice of an arithmetic formula to express consistency;
      \item the choice of a numbering;
      \item the choice of a specific formula numerating the axiom set of $T$.
\end{enumerate}

These factors are not independent of each other, and
a choice made at an earlier stage may have influences on the choices
made at a later stage.
In the following, when we discuss how  $\sf G2$ depends on one factor,  we always assume that other factors are fixed as in the default assumptions we make  and only the factor we are discussing is varied.
For example,  Visser \cite{Visser 11}  rests on fixed choices for (1) and (3)-(5) but varies the choice of (2); Grabmayr \cite{Grabmayr 18} rests on fixed choices for (1)-(2) and (4)-(5) but varies the choice of (3); Feferman \cite{Feferman 60} rests on fixed choices for (1)-(4)   but varies the choice of (5).

In the following, we give a brief discussion of how $\sf G2$ depends on the above five factors.
For more discussions of these factors, we refer to \cite{Cheng 19-2}.

``Whether $\sf G2$ holds for $T$" depends on the choice of the base theory.
An foundational question about $\sf G2$ is: how much of information about arithmetic is required for the proof of $\sf G2$.
If the base theory does not contain enough information about arithmetic,  then $\sf G2$ may fail in the sense that the consistency statement is provable in the base theory.
Willard \cite{Willard 06}  explored the generality and boundary-case
exceptions of $\sf G2$ under some base theories.
Willard constructed examples of recursively enumerable  arithmetical theories that couldn't prove the totality of successor function but could prove their own canonical consistency (see \cite{Willard 01,Willard 06}).
Pakhomov \cite{Fedor Pakhomov} defined a theory $H_{<\omega}$ and showed that it proves its own canonical consistency.
Unlike Willard's theories, $H_{<\omega}$ isn't an arithmetical theory but a theory formulated in the
language of set theory with an additional unary function.

``Whether $\sf G2$ holds for $T$" depends on the definition of provability predicate. Recall that $T$ is a recursively axiomatizable consistent extension of $\mathbf{Q}$.
Being a consistency statement is not an absolute concept but a role w.r.t.~ a choice of the provability predicate.
Note that $\sf G2$ holds for any standard provability predicate in the sense that if provability predicate $\mathbf{Pr}_{T}(x)$ is standard, then $T\nvdash \neg \mathbf{Pr}_T(\ulcorner \mathbf{0}\neq \mathbf{0}\urcorner)$.
However, $\sf G2$ may fail for some non-standard provability predicates. Rosser provability predicate is an important kind of non-standard provability predicate in the study of meta-mathematics of arithmetic.
Define the Rosser provability predicate $\mathbf{Pr}^R_{T}(x)$ as the formula $\exists y(\mathbf{Prf}_{T}(x,y)\wedge \forall z\leq y\neg \mathbf{Prf}_{T}(\dot{\neg}(x),z))$.\footnote{$\dot{\neg}$ is a function symbol expressing a primitive recursive function calculating the code of $\neg \phi$ from the code of $\phi$.} Define the consistency statement $\mathbf{Con}^R(T)$ via Rosser provability predicate as $\neg \mathbf{Pr}^R_{T}(\ulcorner \mathbf{0}\neq \mathbf{0}\urcorner)$. Then $\sf G2$ fails for Rosser provability predicate: $T\vdash\mathbf{Con}^R(T)$.

``Whether $\sf G2$  holds for $T$" depends on the choice of arithmetic formulas to express consistency. We have different ways to express the consistency of  $T$.
The canonical arithmetic formula to express the consistency of  $T$ is $\mathbf{Con}(T)\triangleq\neg \mathbf{Pr}_T(\ulcorner \mathbf{0}\neq \mathbf{0}\urcorner)$.
Another way to express the consistency of  $T$ is
$\mathbf{Con}^0(T) \triangleq \forall x(\mathbf{Fml}(x) \wedge \mathbf{Pr}_T(x) \rightarrow \neg \mathbf{Pr}_T(\dot{\neg} x))$.\footnote{$\mathbf{Fml}(x)$ is the formula which represents the relation that $x$ is a code of a formula.}

Kurahashi \cite{Rosser provability and G2} constructed a Rosser provability predicate such that $\sf G2$ holds for the consistency statement formulated via $\mathbf{Con}^0(T)$ (i.e.~ the consistency statement formulated via $\mathbf{Con}^0(T)$ and Rosser provability predicate is not provable in $T$), but $\sf G2$ fails for the consistency statement formulated via $\mathbf{Con}(T)$ (i.e.~ the consistency statement formulated via $\mathbf{Con}(T)$ and Rosser provability predicate is  provable in $T$).

``Whether $\sf G2$ holds for $T$" depends on the choice of numberings.
Any injective function $\gamma$
from a set of $L(\mathbf{PA})$-expressions to $\omega$ qualifies as a numbering.
G\"{o}del's numbering is a special kind of numberings under which the G\"{o}del number of the set of axioms of $\mathbf{PA}$ is recursive.
Grabmayr \cite{Grabmayr 18} showed that $\sf G2$ holds for acceptable numberings; But $\sf G2$  fails for some non-acceptable numberings.\footnote{For the definition of acceptable numberings, we refer to Grabmayr \cite{Grabmayr 18}.}

Finally,
``Whether $\sf G2$ holds for $T$" depends on the numeration of $T$.  As a generalization, $\sf G2$ holds for any $\Sigma^0_1$ numeration of $T$: if $\alpha(x)$ is a $\Sigma^0_1$  numeration of $T$, then $T\nvdash \mathbf{Con}_{\alpha}(T)$.
 However, $\sf G2$ fails for some $\Pi^0_1$ numerations of $T$.
For example, Feferman \cite{Feferman 60} constructed a $\Pi^0_{1}$ numeration $\tau(u)$ of $T$ such that $\sf G2$ fails under this numeration: $T\vdash \mathbf{Con}_{\tau}(T)$.


\begin{thebibliography}{100}




\bibitem{Beklemishev 45}
L. D. Beklemishev. G\"{o}del incompleteness theorems and
the limits of their applicability I.
\emph{Russian Math Surveys}, 2010.



\bibitem{Benacerraf 67} Paul Benacerraf. God, the devil and G\"{o}del. \emph{The Monist}  51,
pp. 9-32, 1967.



\bibitem{Boolos 93}
George Boolos. \emph{The Logic of Provability}. Cambridge University Press, 1993.


\bibitem{Carlson 00} Timothy J. Carlson. Knowledge, machines, and the consistency of Reinhardt's strong mechanistic thesis. \emph{Annals of Pure and Applied Logic},
105 (1-3):51-82, 2000.


\bibitem{Chalmers 95} David J. Chalmers. Minds, machines, and mathematics: A review of
Shadows of the mind by Roger Penrose. \emph{Journal Psyche}, 2, June 1995.


\bibitem{Cheng book 19} Yong Cheng. \emph{Incompleteness for Higher-Order Arithmetic: An Example Based on Harrington's Principle}. Springer series: Springerbrief in Mathematics, Springer, 2019.


\bibitem{Cheng 19} Yong Cheng. Finding the limit of incompleteness I. Accepted and to appear in \emph{The Bulletin of Symbolic Logic}, arXiv:1902.06658v2, 2020.

\bibitem{Cheng 19-2} Yong Cheng. Current research on G\"{o}del's incompleteness theorem, to appear in \emph{The Bulletin of Symbolic Logic}, 2020.




\bibitem{Enderton 2001} Herbert B. Enderton. \emph{A mathematical introduction to logic} (2nd ed.). Boston, MA: Academic Press, 2001.

\bibitem{Epstein 2011}
Richard L. Epstein (with contributions by Les{\l}aw W.Szczerba). \emph{Classical mathematical logic: The semantic foundations of logic}.  Princeton University Press, 2011



\bibitem{Feferman 60} Solomon Feferman. Arithmetization of metamathematics in a general setting. \emph{Fundamenta
Mathematicae} 49: 35-92, 1960.


\bibitem{Feferman 1995} Solomon Feferman. Penrose's G\"{o}delian argument: A review of shadows
of the mind by Roger Penrose. \emph{Journal Psyche}, 2, May 1995.

\bibitem{Feferman 2009}
Solomon Feferman. G\"{o}del, Nagel, minds, and machines. \emph{The Journal of
Philosophy}, CVI(4): 201-219, April 2009.



\bibitem{Gaifman 2000} Haim Gaifman. What G\"{o}del's incompletness result does and does not
show. \emph{The Journal of Philosophy}, XCVII(8): 462-470, August 2000.

\bibitem{Godel 1931 original proof}  Kurt G\"{o}del. \"{U}ber formal unentscheidbare S\"{a}tze der Principia Mathematica und
verwandter Systeme I. \emph{Monatsh. Math. Phys}. 38:1 (1931), 173-198.





\bibitem{interpretation}
Kurt G\"{o}del. An interpretation of the intuitionistic propositional calculus. In \emph{Collected Works, Volume I: Publications 1929-1936}, pp. 301-303. Edited by Solomon Feferman, John W. Dawson,
Jr., Stephen C. Kleene, Gregory H. Moore, Robert M. Solovay, and Jean
van Heijenoort.  Oxford
University Press, 1986.




\bibitem{Godel 1951} Kurt G\"{o}del. Some basic theorems on the foundations of mathematics and their implications. In \cite{Godel Collected Works III}, pp. 304-323, Oxford
University Press, 1951.



\bibitem{Godel Collected Works III}
Kurt G\"{o}del. \emph{Collected Works, Volume III: Unpublished Essays and Lectures}. Edited
by Solomon Feferman, John W. Dawson, Jr., Warren Goldfarb, Charles
Parsons, and Robert M. Solovay. Oxford University Press, New York and Oxford, 1995.


\bibitem{Grabmayr 18} Balthasar Grabmayr. On the Invariance of G\"{o}del's Second Theorem
with regard to Numberings. To appear in \emph{The Review of Symbolic Logic}, 2020.

\bibitem{Metamathematics of First-Order
Arithmetic}
P. H\'{a}jek  and P. Pudl\'{a}k.  \emph{Metamathematics of First-Order
Arithmetic}. Springer-Verlag, Berlin-Heidelberg-New York, 1993.

\bibitem{Halbach 2014a} V. Halbach and A. Visser.   Self-reference in arithmetic I (2014a). \emph{Review of
Symbolic Logic} 7(4), 671-691.

\bibitem{Halbach 2014b} V. Halbach and A. Visser.  Self-Reference in Arithmetic II (2014b). \emph{Review of
Symbolic Logic} 7(4), 692-712.



\bibitem{defense of epistemic arithmetic}
Leon Horsten. In defense of epistemic arithmetic. \emph{Synthese}, 116(1): 1-25, 1998.


\bibitem{DST} Leon Horsten and Philip Welch. \emph{G\"{o}del's Disjunction: The scope and limits of mathematical knowledge},  Oxford University Press 2016.









\bibitem{Peter 06} Peter Koellner. On the Question of Absolute Undecidability. \emph{Philosophia Mathematica}, Vol. 14, No. 2, 2006, pp. 153-188.


\bibitem{Peter 16} Peter Koellner. \emph{G\"{o}del's Disjunction}. Charpter in G\"{o}del's Disjunction: The scope and limits of mathematical knowledge, Edited by Leon Horsten and Philip Welch, Oxford University Press 2016.

\bibitem{Peter 18} Peter Koellner. On the Question of whether the Mind Can Be Mechanized. Part 1: From G\"{o}del to Penrose. \emph{Journal of Philosophy}, Volume 115, Issue 7, Pages 337-360, July 2018.



\bibitem{Peter 18 second} Peter Koellner. On the Question of Whether the Mind Can Be Mechanized. Part 2: Penrose's New Argument.  \emph{Journal of Philosophy}, Volume 115, Issue 9, Pages 453-484, September 2018.



\bibitem{Kotlarski 2004} Henryk Kotlarski.
The incompleteness theorems after 70 years. \emph{Annals of Pure and Applied Logic} 126, 125-138, 2004.

\bibitem{Krajewski 2019} Stanislaw Krajewski.
On the Anti-Mechanist Arguments Based on G\"{o}del Theorem, accepted and to appear in a special issue of \emph{Semiotic Studies}, 2019.


\bibitem{Rosser provability and G2} Taishi Kurahashi. Rosser provability and the second incompleteness
theorem. Reprint,  arXiv:1902.06863, 2019.


\bibitem{On formal and informal provability} Hannes Leitgeb. On formal and informal provability. In Ot\'{a}vio Bueno
and {\O}ystein Linnebo, editors, \emph{New Waves in Philosophy of Mathematics}, pages 263-99. Palgrave Macmillan, 2009.

\bibitem{Per 97} Per Lindstr\"{o}m. \emph{Aspects of Incompleteness}. Lecture Notes in Logic v. 10, 1997.

\bibitem{Per 01} Per Lindstr\"{o}m. Penrose's new argument. \emph{Journal of Philosophical Logic},
30: 241-250, 2001.

\bibitem{Per 06} Per Lindstr\"{o}m. Remarks on Penrose's new argument. \emph{Journal of
Philosophical Logic}, 35: 231-237, 2006.


\bibitem{Lucas 61} J. R. Lucas. Minds, machines, and G\"{o}del. \emph{Philosophy} 36 (1961),
pp. 120-124.



\bibitem{Lucas 96} J. R. Lucas. Minds, machines, and G\"{o}del: A retrospect. \emph{Machines and thought: The legacy
 of Alan Turing}, Volume 1 (P. J. R. Millican and A. Clark, editors), Oxford University Press,
 Oxford, 1996.




\bibitem{Richard Montague} Richard Montague. Syntactical treatments of modality, with corollaries
on reflexion principles and finite axiomatizability. \emph{Acta Philosophica
Fennica}, (16): 153-167, 1963.

\bibitem{metamathematics} Roman Murawski. \emph{Recursive Functions and Metamathematics: Problems of Completeness and Decidability, G\"{o}del's Theorems}. Springer Netherlands, 1999.


\bibitem{Some remarks on the notion of proof}
John Myhill. Some remarks on the notion of proof. \emph{Journal of Philosophy}, LVII(14): 461-471, July 1960.




\bibitem{Godels Proof 2} Ernest Nagel and James R. Newman. \emph{G\"{o}del's Proof}. New York University Press, 2001. Revised edition, 2001.


\bibitem{Fedor Pakhomov} Fedor Pakhomov. A weak set theory that proves its own consistency. Reprint,  arXiv:1907.00877v2, 2019.



\bibitem{Pudlak 99} P. Pudl\'{a}k, A note on applicability of the incompleteness theorem to human mind, \emph{Annals of Pure and Applied Logic} 96 (1999), pp. 335-342.


\bibitem{Penrose 89} Roger Penrose.  \emph{The Emperor's New Mind: Concerning Computeres,
Minds, and the Laws of Physics}. Oxford University Press
(1989).







\bibitem{Penrose 94}  Roger Penrose.  \emph{Shadows of the Mind: A Search for the Missing Science
of Consciousness}. Oxford University Press, 1994.

\bibitem{Penrose 11} Roger Penrose. G\"{o}del, the mind, and the laws of physics. In \emph{Kurt G\"{o}del and the Foundations of Mathematics: Horizons of Truth}, chapter 16, pages 339-358. Edited by Matthias
Baaz, Christos H. Papadimitriou, Hilary W. Putnam, Dana S. Scott, and
Charles L. Harper,  Cambridge University Press, 2011.

\bibitem{Minds and machines} H. Putnam. Minds and machines. \emph{Dimensions of mind: A symposium} (Sidney Hood,
editor), New York University Press, New York, 1960, pp. 138-164.






\bibitem{Reid 96}
Constance Reid, \emph{Hilbert}, Springer, ISBN 0-387-94674-8, 1996.


\bibitem{Reinhardt 32} William N. Reinhardt. Absolute versions of incompleteness theorems.
\emph{No\^{u}s}, 19(3):317-346, September 1985.

\bibitem{Reinhardt 33}  William N. Reinhardt. The consistency of a variant of Church's thesis
with an axiomatic theory of an epistemic notion. In Special Volume for
\emph{the Proceedings of the 5th Latin American Symposium on Mathematical
Logic}, 1981, volume 19 of Revista Colombiana de Matematicas, pages
177-200, 1985.


\bibitem{Reinhardt 86} William N. Reinhardt. Epistemic theories and the interpretation of
G\"{o}del's incompleteness theorems. \emph{Journal of Philosophical Logic}
15 (1986), pp. 427-474.


\bibitem{Shapiro 85} Stewart Shapiro. Epistemic and intuitionistic arithmetic. In Stewart
Shapiro, editor, \emph{Intensional Mathematics}, volume 113 of Studies in Logic
and the Foundations of Mathematics, pages 11-46. North-Holland, 1985.


\bibitem{Incompleteness} Stewart Shapiro. Incompleteness, Mechanism, and Optimism. \emph{The Bulletin of Symbolic Logic}, Vol. 4, No. 3 (Sep., 1998), pp. 273-302.


\bibitem{Shapiro Mechanism} Stewart Shapiro. Mechanism, truth, and Penrose's new argument. \emph{Journal of Philosophical Logic}, XXXII(1):19-42, February 2003.


\bibitem{Smith 2007} Peter Smith. \emph{An Introduction to G\"{o}del's Theorems}. Cambridge University Press, 2007.

\bibitem{Smorynski 1977}
C. Smory\'{n}ski. The Incompleteness Theorems. in: J. Barwise (Ed.), \emph{Handbook of Mathematical Logic},
North-Holland, Amsterdam, 1977, pp. 821-865.


\bibitem{undecidable} {Alfred Tarski, Andrzej Mostowski and Raphael M. Robinson.}
\emph{Undecidabe theories}. Studies in Logic and the Foundations of Mathematics, North-Holland, Amsterdam, 1953.


\bibitem{Thomason} Richmond H. Thomason. A note on syntactical treatments of modality.
\emph{Synthese}, 44: 391-95, July 1980.



\bibitem{Visser 11} Albert Visser. Can we make the second incompleteness theorem coordinate
free? \emph{Journal of Logic and Computation} 21(4), 543-560, 2011.


\bibitem{Visser 16} Albert Visser. The Second Incompleteness Theorem:
Reflections and Ruminations. Chapter in \emph{G\"{o}del's Disjunction: The scope and limits of mathematical knowledge}, Edited by Leon Horsten and Philip Welch, Oxford University Press, 2016.




\bibitem{Hao Wang 96}
Hao Wang. \emph{A Logical Journey: From Godel to Philosophy}. MIT Press, 1996.


\bibitem{Willard 01} D. E. Willard. Self-verifying axiom systems, the incompleteness theorem
and related reflection principles. \emph{Journal of Symbolic Logic}, 66(2): 536-596,
2001.

\bibitem{Willard 06}
D. E. Willard. A generalization of the second incompleteness theorem and
some exceptions to it. \emph{Ann. Pure Appl. Logic}, 141(3): 472-496, 2006.


\bibitem{Williamson 16} Timothy Williamson. Absolute Provability and Safe Knowledge
of Axioms. Chapter in \emph{G\"{o}del's Disjunction: The scope and limits of mathematical knowledge}, Edited by Leon Horsten and Philip Welch, Oxford University Press, 2016.



\bibitem{Richard Zach} Richard Zach. Hilbert's Program Then and Now.
In Dale Jacquette (ed.), \emph{Philosophy of Logic}. Amsterdam: North Holland. pp. 411-447, 2007.


\end{thebibliography}
\end{document}